\documentclass[a4paper,11pt]{amsart}

\usepackage[colorlinks,citecolor=red,urlcolor=blue,bookmarks=false,hypertexnames=true]{hyperref} 
\usepackage{amsmath, amssymb, amsfonts, amsbsy, amsthm, latexsym, epsfig}
\usepackage{epstopdf }
\usepackage{graphicx}
\usepackage[shortlabels]{enumitem}

\usepackage{
color}
\usepackage{lmodern}
\theoremstyle{plain}
\newtheorem{theorem}{Theorem} [section]
\newtheorem{lemma}[theorem]{Lemma}
\newtheorem{proposition}[theorem]{Proposition}

\theoremstyle{definition}

\newtheorem*{teo*}{Theorem}

\theoremstyle{definition}
\newtheorem{definition}[theorem]{Definition}

\newtheorem{remark}[theorem]{Remark}

\newcommand{\esssup}{{\mathrm{ess}\sup}}

\newcommand{\cT}{\mathcal{T}}

\newcommand{\R}{\mathbb{R}}
\newcommand{\Z}{\mathbb{Z}}
\newcommand{\N}{\mathbb{N}}

\newcommand {\w} {\omega}

\begin{document}


\title{Dynamical Sampling for Shift-preserving Operators}

\let\thefootnote\relax\footnote{


The research of the authors is partially supported by grants: UBACyT 20020170100430BA, PICT 2014-1480 (ANPCyT) and CONICET PIP 11220150100355. In particular VP is also supported  by UBACyT 20020170200057BA and PICT-2016- 2616 (Joven).
}

\author{A. Aguilera}
\address{ Departamento de Matem\'atica, Universidad de Buenos Aires,
	Instituto de Matem\'atica "Luis Santal\'o" (IMAS-CONICET-UBA), Buenos Aires, Argentina}
\email{aaguilera@dm.uba.ar}

\author{C. Cabrelli}
\address{ Departamento de Matem\'atica, Universidad de Buenos Aires,
	Instituto de Matem\'atica "Luis Santal\'o" (IMAS-CONICET-UBA), Buenos Aires, Argentina}
\email{cabrelli@dm.uba.ar}

\author{D. Carbajal}
\address{ Departamento de Matem\'atica, Universidad de Buenos Aires,
	Instituto de Matem\'atica "Luis Santal\'o" (IMAS-CONICET-UBA), Buenos Aires, Argentina}
\email{dcarbajal@dm.uba.ar}

\author{V. Paternostro}
\address{ Departamento de Matem\'atica, Universidad de Buenos Aires,
	Instituto de Matem\'atica "Luis Santal\'o" (IMAS-CONICET-UBA), Buenos Aires, Argentina}
\email{vpater@dm.uba.ar}

\begin{abstract}
	In this note, we solve the dynamical sampling problem
	for a class of shift-preserving operators $L:V\to V$ acting on a
	finitely generated shift-invariant space $V$. We find conditions
	on $L$   and  a finite set of functions of $V$ so that
	 the iterations of the operator $L$ on the functions produce
	a frame generator set of $V$. This means that  the integer translations of the
	 generators form a  frame of $V$. 
	\end{abstract}

\maketitle
\section{Introduction}

Dynamical Sampling addresses the problem of recovering  a signal that evolves from its spatial-time samples.
That is, let $D$ be a bounded operator (the  evolution operator) and $f$ an unknown signal that we want to recover.
Assume that we have insufficient samples of $f$. Will it be possible to compensate for this lack of information from $f$,
if we sample the evolved signals at the same locations? i.e. sampling the signals $Df, D^2f,...$?

Mathematically, this question can be reformulated as  follows (see \cite{ACMT}):
 Let $\mathcal{H}$ be a separable Hilbert space, $D: \mathcal{H} \rightarrow \mathcal{H}$
a bounded operator and $\mathcal F=\{f_i: i \in I\}  \subset \mathcal{H}$ a set of functions. Find conditions on $D$ and $\mathcal{F}$ such that
$\{D^j f_i: i\in I,\, j\in K\}$  is a basis or a frame of $\mathcal{H}$. Here, $I$ and $K$  are subsets of $\mathbb{N}\cup \{0\}$.

This problem has recently attracted a lot of attention and has been set in different scenarios.
See   \cite{AADP},\cite{ADK},\cite{AT14},\cite{APT15},\cite{ADK15},\cite{ACMT},\cite{AP} for different instances of the dynamical sampling problem. In \cite{CHP},\cite{CHR},\cite{CHS18},\cite{P17}, the authors studied the problem of when a given frame  can be represented as a discrete orbit of an operator. See also
\cite{ACMP19},\cite{Tang17} for posible applications. 

One case that is mathematically very deep and rich is  when $\mathcal{H}$
is infinite-dimen\-sional, $D$ is a bounded normal operator and $\mathcal F$ is finite.
This case has been tackled using  techniques from different areas such as spectral theory, Hardy spaces,
and Carleson measures (see \cite{ACMT},\cite{AP},\cite{CMPP},\cite{ACCMP},\cite{CMS19}). In particular, it has been proved that in this case the iterations of a 
finite set of functions under a bounded normal operator  will never be a basis  (see \cite{ACMT} and \cite{CMPP}),
so only frames are possible and require many hypotheses on $D$ and $\mathcal F.$
On the other hand, the finite-dimensional case was solved completely for general linear transformations in \cite{ACMT}, (see also \cite{CMPP}).
However, in this case, no estimates  of the frame bounds were given.

In this article, we study the {\it Dynamical Sampling}  problem for {\it shift-preserving} operators acting on {\it shift-invariant}  spaces of $L^2(\R^d)$, as we describe below.

A shift-invariant  space is a  closed subspace $V$ of $L^2(\R^d)$ that is invariant under the action of translations along $\Z^d$ (or a lattice in a more general case). 
A set $\Phi \subset V$ is a set of generators of $V$ if $V=S(\Phi) = \overline{\text{span}} \left\{ T_k\varphi\,:\,\varphi \in \Phi, \,k\in\mathbb Z^d \right\}$, where $T_kf(x)= f(x-k)$.
 If $V=S(\Phi)$ for a finite set $\Phi$ then $V$ is said to be {\it finitely generated},
 and its {\it length} is the minimum  cardinal of all sets of generators of $V$.
 If the integer translates of a set of  generators form a frame of $S(\Phi)$ we say that $\Phi$ is a {\it frame generator set}.
An operator  $L :V \rightarrow V$ is said {\it shift preserving} if it commutes with  integer translates (i.e. $LT_k=T_kL$ for every $k\in \Z^d$).

The dynamical sampling problem  for shift-preserving operators that we study and solve in this article is the following.

Assume that $V$ is a finitely generated shift-invariant  space of length $\ell$ and that $\mathcal F=\left\{ f_1,...,f_m \right\} \subset V$. Let $L: V \rightarrow V$ be a bounded shift-preserving  operator. {\it Find necessary and sufficient conditions  on $\mathcal F$ and $L$ 
 in order that the collection $\left\{L^j f_i : i=1,...,m \,; j=0,...,\ell-1 \right\}$
is a frame generator set of $V$.}

As an application, we can think that we want to recover the shift-invariant space $V$ and we only know some functions $ f_1,...,f_m $ in $V$ (i.e. the set $\mathcal F$).
If $\mathcal F$ is not a set of generators, then the integer  translates of the functions in $\mathcal F$ are not enough to obtain the whole space $V.$
So we resort to the evolution operator $L$ to get a frame generator set.

To solve the dynamical sampling  problem  for shift-preserving operators, we make  intensive use of the concept of {\it range function}, {\it range operator }and  {\it fiberization techniques}, (see Section \ref{sec-preliminaries} for definitions and properties).

The results require a thorough understanding of the structure of shift-preserving  operators. 
The invariant space $V$  is isomorphic to a field  of finite-dimensional subspaces of $\ell^2(\Z^d)$  (each one is a value of the range function) and the shift-preserving operator  is isomorphic to a field of linear transformations  acting on these subspaces (the range operator).

Since each value of a range operator is a linear transformation acting on a finite-dimensional space,
the main idea behind this program is to translate the well known structure of these  linear transformations to our shift-preserving  operator. This requires certain uniformity across the different values of the range operator. 
In this regard, we will turn to the theory of $s$-diagonalization of shift-preserving operators that is fully developed in \cite{ACCP}.

The key point in our analysis will be to reduce the problem of dynamical sampling for a shift-preserving operator acting on a  finitely generated  shift-invariant space,  to a  family of dynamical  sampling problems where the operator that we iterate (the fiber of the range operator) acts on a finite-dimensional space (the fiber spaces).  
Then, we apply the known results on the finite-dimensional dynamical sampling to this case. Since  our problem  requires that the   frame bounds are uniform across the fiber spaces, we obtain an estimate of the frame bounds for the finite-dimensional case, 
(see Section \ref{section-SIS}).

Using this estimate we  obtain our main result that solves the dynamical problem. That is, we give necessary and sufficient conditions on $L$ an $\mathcal F$  to extend the set $\mathcal F$ to a frame generator set by iterating $L$ on the functions in $\mathcal F$.

The paper is organized as follows. In Section \ref{sec-preliminaries} we provide the basics of the theory of shift-invariant spaces and shift-preserving operators needed for the development of the whole paper. In particular, we review the theory of 
$s$-diagonalization for shift-preserving operators.
The main results are presented in Section \ref{section-SIS}.
In Subsection \ref{sec-bounds} we  give frame bound estimates for the finite-dimensional case of dynamical sampling.  Then, in Subsection \ref{section-DS-SP} we solve the dynamical sampling problem for shift-preserving operators.

\section{Preliminaries}\label{sec-preliminaries}

\subsection{Shift-invariant spaces and frame generator sets}\label{subsection-SIS-FGS}

\

In this subsection, we review some of the standard facts on shift-invariant subspaces of $L^{2}(\mathbb{R}^{d})$ and recall a caracterization of frames on these spaces using a technique known as fiberization. 
These spaces have been used in approximation theory, sampling theory, and wavelets, and their structure is very well known.
See \cite{H},\cite{BDR1},\cite{BDR2},\cite{RS},\cite{B} in the Euclidean case, and \cite{CP10},\cite{BR14},\cite{BHP15} in the context of topological groups.
We now state the precise definitions and some properties of these spaces. For properties of frames see \cite{C}.

\begin{definition}
	A closed subspace $V\subset L^2(\mathbb R^d)$ is {\it shift invariant} if for each $f\in V$ we have $T_k f\in V$ for any $k\in \mathbb Z^d$, where $T_kf(x)= f(x-k)$.
\end{definition}

Given a countable set $\Phi\subset L^2(\mathbb R^d)$, the shift-invariant space generated by $\Phi$ is
$$V=S(\Phi) = \overline{\text{span}} \left\{ T_k\varphi\,:\,\varphi \in \Phi, \,k\in\mathbb Z^d \right\}$$
and $\Phi$ is called a set of generators of $V$. When $\Phi$ is a finite set, we say that $V$ is a {\it finitely generated} shift-invariant space. 

We will denote by $E(\Phi)$ the family of translations of $\Phi$, i.e. 
$$E(\Phi) = \left\{T_k\varphi\,:\,\varphi \in \Phi, \,k\in\mathbb Z^d \right\}.$$
When $E(\Phi)$ forms a frame of V we will
say that $\Phi$ is a {\it frame generator set} of V.

The structure of a shift-invariant space can be studied in terms of its range function. This technique is known as fiberization.

Let $L^2([0,1)^d,\ell^2(\mathbb Z^d))$ be the Hilbert space of all vector-valued measurable functions $\psi: [0,1)^d\rightarrow \ell^2(\mathbb Z^d)$ with finite norm, where the norm is given by $$\|\psi\|=\left( \int_{[
0,1)^{d}
} \|\psi(\omega)\|^2_{\ell^2} \,d\omega\right)^{1/2}.$$ 

The Fourier transform  of $f\in L^1(\R^d)$ is given by
$$\hat{f}(\w) = \int_{\mathbb{R}^{d}} f(x)e^{-2\pi i \langle x,\w\rangle}\,dx, $$
and extends by density to an isometric isomorphism in $L^2(\R^d)$.
\begin{proposition} \cite[Proposition 1.2]{B}\label{isometria}
	The map $\mathcal T: L^2(\mathbb R^d) \to L^2([0,1)^d,\ell^2(\mathbb Z^d))$ defined by
	$$
	\mathcal T f(\omega) = \{\hat{f} (\omega+k) \}_{k \in \mathbb Z^d},
	$$
	is an isometric isomorphism. We call $\mathcal T f(\omega)$ the fiber of $f$ at $\omega$. Moreover, it satisfies that \begin{equation*}\label{modulation}
	\mathcal{T} T_{k}f(\omega) = e_k(\omega) \,\mathcal T f(\omega),
	\end{equation*}   
	where $e_k(\omega)=e^{-2\pi i\langle \omega,k\rangle}$, $k\in\mathbb{Z}^{d}$.
\end{proposition}

\begin{definition}
A {\it range function} is a mapping
\begin{align*}
J:[0,1)^d&\rightarrow\{\text{\,closed subspaces of }\ell^2(\mathbb Z^d)\,\}\\
\omega&\mapsto J(\omega).
\end{align*}
\end{definition}

It is said that a range function $J$ is measurable if the scalar function $\omega\mapsto\langle P_{J(\omega)}u,v \rangle$ is measurable for every $u,v\in\ell^2(\mathbb Z^d)$, where $P_{J(\omega)}$ is the orthogonal projection from $\ell^2(\mathbb Z^d)$ onto $J(\omega)$.

Shift-invariant spaces are characterized in terms of range functions as the following theorem, due to Bownik, shows.

\begin{theorem}\cite[Proposition 1.5]{B}
A closed subspace $V\subset L^{2}(\mathbb{R}^{d})$ is shift invariant if and only if there exists a measurable range function $J$ such that
\begin{equation*}
V=\{f\in L^{2}(\mathbb{R}^{d}): \mathcal{T}f(\omega) \in J(\omega) \text{ for a.e. } \omega \in [0,1)^{d}\}.
\end{equation*}
Furthermore, if $V=S(\Phi)$ for some coun\-table set $\Phi\subset L^2(\mathbb R^d)$, then 
\begin{equation*}
J(\omega) = \overline{\text{span}}\{\mathcal T f(\omega)\,:\,f\in\Phi\,\}
\end{equation*} 
for a.e. $\omega\in[0,1)^d$.
\end{theorem}

We call the subspace $J(\omega)$ the fiber space of $V$ at $\omega$. Under the convention that two range functions are identified if they are equal a.e. $\omega\in [0,1)^{d}$, the correspondence between $V$ and $J$ is one-to-one.

In particular, when $\Phi$ is a finite set, the previous theorem allows us to translate some problems in infinite-dimensional shift-invariant spaces, into problems of finite dimension that can be treated with linear algebra.

The next lemma was proved by Helson (see \cite{H}) and will be useful in the main results of this paper.
\begin{lemma}\label{Helson projections}
Let $V\subset L^{2}(\mathbb{R}^d)$ be a shift-invariant space. For each $f\in L^2(\mathbb R^d)$ we have that
$$\mathcal T (P_V f)(\omega) = P_{J (\omega)}(\mathcal T f (\omega)).$$
\end{lemma}

It is possible to give a characterization of frames of a shift-invariant space $V$ in terms of its fibers as shown in the next result which will be crucial for our problem. 

\begin{theorem} \cite[Theorem 2.3]{B}\label{riesz basis fiber}
	Let $\Phi\subset L^2(\mathbb R^d)$ be a countable set. Then the following conditions are equivalent:
	\begin{enumerate}[\rm(i)]
	\item The system $E(\Phi)$ is a frame of $V$ with bounds $A,B>0$;
	\item The system $\left\{\mathcal T \varphi(\omega)\,:\,\varphi\in\Phi\,\right\}\subset\ell^2(\mathbb Z^d)$ is a frame of $J(\omega)$ with uniform bounds $A,B>0$ for a.e. $\omega\in [0,1)^d$.
	\end{enumerate}
\end{theorem}

Recall that the {\it length} of a finitely generated shift-invariant space $V\subset L^{2}(\mathbb{R}^{d})$ is 
defined as the smallest natural number $\ell$ such that there exist $\varphi_{1},...,\varphi_{\ell} \in V$ with $V=S(\varphi_{1},...,\varphi_{\ell})$.
An equivalent definition of the length in terms of the range function $J$ associated to $V$ is $\ell= \text{\rm ess sup}_{\omega\in[0,1)^d} \dim J(\omega)$.

The {\it spectrum} of $V$ is defined by the set $$ \sigma(V) = \left\{\omega \in [0,1)^d\,:\, \dim J(\omega) > 0 \right\}.$$

\

\subsection{Shift-preserving operators and s-diagonalization}\label{subsection-seigenvalues} 

In this subsection, we give a brief exposition on the structure of shift-preserving operators and $s$-dia\-go\-na\-li\-za\-tion. For a more comprehensive treatment of this topic we refer the reader to \cite{ACCP} and \cite{B}.

\begin{definition}
Let $V\subset L^2(\mathbb R^d)$ be a shift-invariant space and $L:V\rightarrow L^2(\mathbb R^d)$ be a bounded operator. We say that $L$ is {\it shift preserving} if $LT_k = T_k L$ for all $k\in\mathbb Z^d$.
\end{definition}

Shift-preserving operators are the natural operators acting on shift-invariant\break spaces.
They were introduced by Bownik in \cite{B}, where he also studied their  properties through the concept of {\it range operator}.
 The notion of range operator permits to decode the action of a shift-preserving  operator  throughout
its  fiber map. These last two concepts are a very natural tool for studying shift-invariant spaces.
Shift-preserving operators are in one to one correspondence with range operators. 
\begin{definition}
Let $V$ be a shift-invariant space with range function $J$. A {\it range operator} on $J$ is a mapping
$$R: [0,1)^d\rightarrow \left\{ \text{bounded operators defined on closed subspaces of } \ell^2(\mathbb Z^d) \right\},$$
such that the domain of $R(\omega)$ is $J(\omega)$ for a.e. $\omega\in[0,1)^d$. 
\end{definition}

It is said that $R$ is measurable if  $\omega \mapsto \langle R(\omega) P_{J(\omega)} u,v\rangle$ is a measurable scalar function for every $u,v\in\ell^2(\mathbb Z^d)$.

    In \cite[Theorem 4.5]{B}, Bownik proved that given a shift-preserving operator $L:V\rightarrow L^2(\mathbb R^d)$, there exists a measurable range operator $R$ on $J$ such that 
\begin{equation}\label{prop of R}
(\mathcal T\circ L) f(\omega) = R(\omega) \left(\mathcal T f(\omega)\right),
\end{equation}
for a.e. $\omega\in\mathbb [0,1)^d$ and $f\in V$. If range operators which are equal almost everywhere are identified, the range operator associated to $L$ is unique.

Conversely, if $R$ is a measurable range operator on $J$ with $$\underset{\omega\in [0,1)^d}{\esssup} \|R(\omega)\|<\infty $$ then, there exists a bounded shift-preserving operator $L:V\rightarrow L^{2}(\R^{d})$ such that the intertwining property (\ref{prop of R}) holds, and its norm operator is given by
\begin{equation}\label{norm of R and L}
\|L\| = \underset{\omega\in [0,1)^d}{\esssup} \|R(\omega)\|.
\end{equation}

We will consider the particular case where $L:V\rightarrow V$. Thus, this implies that $L$ has a corresponding range operator $R$ such that $R(\omega):J(\omega)\rightarrow J(\omega)$ for a.e. $\omega\in[0,1)^d$. The following result is due to Bownik (see \cite{B}).

\begin{theorem}\label{adjoint L}
Let $V$ be a shift-invariant space and $L:V\rightarrow V$ a shift-preserving operator with associated range operator $R$. Then, the adjoint operator $L^*:V\rightarrow V$ is also shift preserving and its corresponding range operator $R^*$ satisfies that $R^*(\omega) = (R(\omega))^*$ for a.e. $\omega\in [0,1)^d$.  As a consequence,  $L$ is self-adjoint if and only if  $R(\omega)$ is self-adjoint for a.e. $\omega \in [0,1)^d$, and  $L$ is a normal operator if and only if  $R(\omega)$ is a normal operator for a.e. $\omega \in [0,1)^d$.  
\end{theorem}

The structure of a shift-preserving operator $L$ acting on a shift-invariant space $V$ can be studied in terms of its range operator, as suggested by \eqref{prop of R}. When $V$ is a finitely generated shift-invariant space, $R$ can be seen as a field of linear transformations acting on finite-dimensional spaces. We can exploit this fact to translate the structure of these linear transformations back to the shift-preserving operator $L$ through the isometric isomorphism $\mathcal{T}$. In particular, this correspondence between $L$ and $R$ allows us to introduce a notion of diagonalization for shift-preserving operators, which is called $s$-diagonalization. This theory has been extensively developed in \cite{ACCP}. Here, we present the main definitions and results. For instance, we have a simpler representation of a bounded, normal, shift-preserving operator.

\begin{definition}
	We say that a sequence $a=\{a(j)\}_{j\in\Z^d}\in \ell^{2}(\mathbb{Z}^d)$ is of {\it bounded spectrum} if $\hat{a} \in L^{\infty}([0,1)^{d})$, where $\hat{a}(\omega)= \sum\limits_{j\in \mathbb{Z}^{d}} a(j)e_j(\omega)$.
\end{definition}

Now, we give the definitions of $s$-eigenvalue and $s$-eigenspace of a shift-preserving operator.

\begin{definition}\label{s-eigenvalue and s-eigenspace}
  Let $V$ be a shift-invariant space and $L:V\rightarrow V$ a bounded shift-preserving operator. Given $a\in \ell^{2}(\mathbb{Z}^d)$ a sequence of bounded spectrum, let $\Lambda_a:V\rightarrow V$ be the operator defined by $\Lambda_a = \sum_{j\in\mathbb Z^d} a(j)\, T_j.$  We say that $\Lambda_a$ is an {\it $s$-eigenvalue} of $L$ if 
  \begin{equation*} 
  	V_a := \ker\left(L - \Lambda_a\right)\neq\{0\}.
  	\end{equation*}
  	We call $V_a$ the {\it $s$-eigenspace} associated to $\Lambda_a$.
 \end{definition} 	
  
  It is easy to see that the bounded spectrum assumption of $a$ guarantees that $\Lambda_a$ is well-defined and bounded (see \cite[Proposition 4.1]{ACCP}). Furthermore, $V_a$ is a shift-invariant subspace of $V$ and for every $f\in V_a$, we have that $Lf=\Lambda_af.$ Since $\widehat{\Lambda_{a}f}=\hat{a}\hat{f}$,  we see that for every $f\in V_a$,
 \begin{equation*}\label{intertwining-prop} 
 R(\omega)(\mathcal Tf(\omega)) = \mathcal T(Lf)(\omega) = \cT(\Lambda_a f)(\w) = \hat{a}(\w)\cT f(\w),
 \end{equation*}
 for a.e. $\omega\in [0,1)^d.$ This shows that the $s$-eigenvalues of $L$ are closely related with the eigenvalues of its range operator, as we state in the next result whose proof is given in \cite{ACCP}.
 
 \begin{proposition}\label{prop-eigen}
 Let $V$ be a shift-invariant space with range function $J$ such that $\dim J(\w)<\infty$ for a.e. $\w\in [0,1)^d$, $L:V\rightarrow V$ a bounded shift-preserving operator and  $a\in \ell^{2}(\mathbb{Z}^d)$ a sequence of bounded spectrum. Then, the following statements hold:
 \begin{enumerate}
\item[\rm(i)] \label{eigenvalues}If $\Lambda_a$ is an $s$-eigenvalue of $L$, then $\lambda_a(\omega) := \hat{a}(\omega)$ is an eigenvalue of $R(\omega)$ for a.e. $\omega\in \sigma(V_a)$.
\item[\rm(ii)]  \label{eigenspaces} The mapping $\omega \mapsto \ker\left(R(\omega) - \lambda_a(\omega)\mathcal I\right)$, $\omega \in [0,1)^d$ is the measurable range function of $V_a$, which we will denote $J_{V_a}$.
 \end{enumerate}
 \end{proposition}
 
We remark that if $V_a\cap V_b = \{0\}$, then $\hat{a}(\w)\neq \hat{b}(\w)$ almost everywhere in \break $\sigma(V_a)\cap\sigma(V_b)$, that is, $\hat{a}(\w)$ and $\hat{b}(\w)$ correspond to different eigenvalues of $R(\w)$ (see \cite[Proposition 4.7]{ACCP}).

 \begin{definition}
 	Let $V$ be a finitely generated shift-invariant space and $L:V\rightarrow V$ a bounded shift-preserving operator. We say that $L$ is {\it $s$-diagonalizable} if there exist $r\in\mathbb{N}$ and $a_1,\dots,a_r$ sequences of bounded spectrum  such that $\Lambda_{a_1}, \dots, \Lambda_{a_r}$ are $s$-eigenvalues of $L$  and
 	\begin{equation*}\label{direct sum V}
		V = V_{a_1}\oplus \dots \oplus V_{a_r},
	\end{equation*}
	where $\Lambda_{a_j}$ and $V_{a_j}$ for $j=1,...,r$ are given in Definition \ref{s-eigenvalue and s-eigenspace}. In this case, we say that $(V,L,a_1,...a_r)$ is an \it{$s$-diagonalization} of $L$.
 \end{definition}
 
When $L$ is $s$-diagonalizable and $(V,L,a_1,...a_r)$ is an $s$-diagonalization of $L$, then $R(\omega)$ is diagonalizable for a.e. $\omega\in\sigma(V)$ (see  \cite[Theorem 6.4]{ACCP}). In particular, the range function $J$ associated to $V$ has the following decomposition in direct sum 
\begin{equation*}\label{direct sum J}
		J(\omega) = J_{V_{a_1}}(\w)\oplus \dots \oplus J_{V_{a_r}}(\w),
	\end{equation*}
for a.e. $\omega\in [0,1)^{d}$,
where $J_{V_{a_s}}(\omega) = \ker(R(\omega)-\hat{a}_{s}(\omega))$ is the range function associated to the shift-invariant subspace $V_{a_s}$ for every $s=1,...,r$.

\begin{remark}\label{Teo 6.13}
In \cite[Theorem 6.13]{ACCP}, it was proved that given an $s$-diagonalizable shift-preserving operator $L$  acting on a finitely generated shift-invariant space $V$, there always exists an $s$-diagonalization of $L$, say $(V,L,a_1,...,a_r)$, with the property that the spectra of the $s$-eigenspaces of $L$ satisfy the condition $\sigma(V_{a_{s+1}})\subseteq \sigma(V_{a_{s}})$ for every $s=1,...,r-1$.

 We define for $s=1,\dots,r-1$ the sets $B_s := \sigma(V_{a_s})\setminus \sigma(V_{a_{s+1}})$ and $B_r:=\sigma(V_{a_r})$. Then,  $\sigma(V)=\bigcup_{s=1}^{r}B_s$ where the union is disjoint. Given $s\in\{1,\dots,r\}$, notice that for a.e. $\w\in B_s$ we have that $J_{V_{a_j}}(\w)\neq \{0\}$ for $1\leq j\leq s$ and $J_{V_{a_j}}(\w) = \{0\}$ for $s<j\leq r$, i.e. $R(\w)$ has exactly $s$ different eigenvalues in $B_s$.

Moreover, the construction of such $s$-diagonalization is based on \cite[Theorem 6.8]{ACCP} where the $s$-eigenvalues obtained satisfy that
$$\hat{a}_s(\w) = \begin{cases}
\lambda_s(\w),\,&\w\in \sigma(V_a)\\
K +s & \text{otherwise}
\end{cases}, \quad \w\in[0,1)^d$$
where $\lambda_s(\w):\sigma(V_{a_s})\to\mathbb C$ is a measurable function which is an eigenvalue of $R(\w)$ almost everywhere in $\sigma(V_{a_s})$ and $K>0$ is a constant such that $ K \geq \|R(\w)\|$ for a.e. $\w\in[0,1)^d$.

This facts will be needed in Subsection \ref{section-DS-SP}.

\end{remark}

The following generalized Spectral Theorem was obtained in \cite{ACCP} for bounded, shift-preserving operators which are normal.
 
\begin{theorem}\label{spectral theorem}
	Let $V$ be a finitely generated shift-invariant space and  $L:V\rightarrow V$ a bounded shift-preserving operator. If $L$ is normal, then it is $s$-diagonalizable and, if $(V,L,a_1, \dots,a_r)$ is an $s$-diagonalization of $L$, we have that
	\begin{equation*}\label{spectral theorem L}
		L= \sum_{s=1}^{r} \Lambda_{a_s} P_{V_{a_s}},
	\end{equation*}
where $P_{V_{a_s}}$ denotes the orthogonal projection of $L^2(\R^d)$ onto $V_{a_s}$ for $s=1, \dots, r$.
\end{theorem}

Finally, we add the next proposition which relates the $s$-diagonalization of a normal operator and its adjoint. 

\begin{proposition}\label{adjoint diag}
	Let $V$ be a finitely generated shift-invariant space and  $L:V\rightarrow V$ a bounded shift-preserving operator. 	If $L$ is normal, the following statements hold:
	\begin{enumerate}[\rm (i)]
		\item $L$ and $L^*$ are $s$-diagonalizable.
    		\item If $\Lambda_a$ is an $s$-eigenvalue of $L$, then $\Lambda_a^*$ is an $s$-eigenvalue of $L^*$ and $\ker(L^*-\Lambda_a^*)=V_a$. Furthermore, $\Lambda_a^*= \Lambda_b,$  where $b\in\ell^2(\Z^d)$ is defined by $b(j):=\overline{a(\text{-} j)}$ for $j\in \Z^d$.
		\item If $(V,L,a_1,\dots,a_r)$ is an $s$-diagonalization of $L$, then $(V,L^*,b_1,\dots,b_r)$ is an $s$-diagonalization of $L^*$, where $\Lambda_{b_s}=\Lambda_{a_s}^*$ for $s=1,\dots,r$.
	\end{enumerate}   
\end{proposition}

\begin{proof}
	Since $L$ is normal, so is $L^*$ and by Theorem \ref{spectral theorem}, $L$ and $L^*$ are $s$-diagonalizable, which proves (i).
	
	Furthermore, assume that $\Lambda_a$ is an $s$-eigenvalue of $L$. Let $f,g\in V$, we have that
	\begin{equation}\label{adjoint Lambda_a}
		\left\langle \Lambda_a f, g\right\rangle = \left\langle \hat{a}\hat{f},\hat{g}\right\rangle = \left\langle \hat{f}, \overline{\hat{a}}\hat{g}\right\rangle.
	\end{equation}
	On the other hand, it is easy to see that $$\overline{\hat{a}}(\w)=\sum_{j\in\Z^d}\overline{a(\text{-} j)}e_j(\w).$$
	Thus, if we define $b\in\ell^2(\Z^d)$ by $b(j):=\overline{a(\text{-} j)}$ for $j\in \Z^d$, we deduce from \eqref{adjoint Lambda_a} that $\left\langle \Lambda_a f, g\right\rangle = 	\left\langle f, \Lambda_b g\right\rangle,$
	that is $\Lambda_a^* = \Lambda_b$.
	
	On the other hand, $L-\Lambda_a$ is a normal operator and so $\{0\}\neq V_a = \ker(L-\Lambda_a) = \ker\left((L-\Lambda_a)^*\right) = \ker \left(L^*-\Lambda_b\right)$, from which we conclude that $\Lambda_b$ is an $s$-eigenvalue of $L^*$, as we wanted to see in (ii).
	
	For (iii), it only remains to observe that if $(V,L,a_1,\dots,a_r)$ is an $s$-diagonalization of $L$, then we have that $V=V_{a_1}\oplus\dots\oplus V_{a_r}$, which is also a decomposition into $s$-eigenspaces of $L^*$ since $\ker(L^*-\Lambda_{b_s})=V_{a_s}$ for every $s=1,\dots,r$ by item (ii).
\end{proof}

\section{A Dynamical Sampling problem for shift-preserving operators}\label{section-SIS}

Given a set of functions $\mathcal F= \{f_i\,:\,i\in I\}$ in a Hilbert space $\mathcal H$ and a bounded operator $D:\mathcal H\to\mathcal H$, the dynamical sampling problem consists on finding conditions on $D$ and $\mathcal F$ for $\{D^{j}f_i\,:\,i\in I, j\in K\}$ to be a frame of $\mathcal H$, where $I,K\subseteq\N\cup\{0\}$.  

The finite-dimensional case of this problem has been solved in \cite{ACMT}. Whereas, for the infinite-dimensional case, a characterization theorem was given for normal operators iterating a finite set of functions in \cite{ACMT} and \cite{CMPP}. Moreover, the following condition  was also proved in \cite{CMPP} for the finite-dimensional setting. 
\begin{theorem}\label{DS finite dimension}
	Let $\mathcal H$ be an $n$-dimensional Hilbert space and let $D:\mathcal H\to\mathcal H$ be a linear transformation. Let $I=\{1,\dots,m\}$, $K=\{0,\dots,n-1\}$ and let $\{f_i\,:\,i\in I\}\subseteq \mathcal H$. Then, $\{D^{j}f_i\,:\,i\in I, j\in K\}$ is a frame of $\mathcal H$ if and only if for each $\lambda\in\text{spec}(D^*)$, $$\left\{P_{\ker(D^*-\lambda \mathcal I)}f_i\,:\,i\in I\right\}$$
	is a frame of $\ker(D^*-\lambda \mathcal I)$. 
 Here, $\text{spec}(D^*)$ denotes the set of eigenvalues of $D^*$.
\end{theorem}

In what follows, we formulate a dynamical sampling problem for shift-preserving operators. Let $V$ be a finitely generated shift-invariant space whose length is $\ell$. Suppose that $I=\{1,\dots,m\}$ is a finite index set and $K=\{0,\dots,\ell-1\}$. Let $\{f_i\,:\,i\in I\}$ be a set of functions in $V$ and let $L:V\to V$ be a bounded shift-preserving operator. We want to give conditions for the system
\begin{equation}\label{iteration set}
\left\{L^{j}f_i\,:\,i\in I, j\in K\right\}
\end{equation} to be a frame generator set of $V$. 

Our approach on solving this problem is through fiberization techniques. Let $J$ be the range function of $V$ and let $R$ be the associated range operator of $L$. Recall that by Theorem \ref{riesz basis fiber}, the system in \eqref{iteration set} is a frame generator set of $V$ if and only if its fibers form a frame of $J(\w)$ with uniform bounds for a.e. $\w\in [0,1)^d$. On the other hand, by the intertwinig property of $\mathcal T$ and $R$  given in equation \eqref{prop of R}, we have that $\mathcal T(L^jf_i)(\w)=R(\w)^j(\mathcal Tf_i(\w))$. 

This allows us to reduce our problem to a finite-dimensional dynamical sampling problem by studying conditions for the system
\begin{equation*}
\left\{R(\omega)^{j}(\mathcal{T}f_{i}(\omega)): i\in I, \,j\in K\right\}
\end{equation*}
to be a frame of $J(\w)$ for a.e. $\w\in[0,1)^d$. For this, one would like to apply Theorem \ref{DS finite dimension} with $D=R(\w)$ at almost every $\w\in[0,1)^d$. To translate back results from the range operator to the shift-preserving operator $L$, we need uniformity in the frame bounds. However, in Theorem \ref{DS finite dimension} frame bounds estimates were not provided. 
For this, we need to obtain frame bounds for the set of iterations in terms of the frame bounds 
of the frame of the projections of the elements of $\mathcal F$ and reciprocally.

We divide this section in two subsections. The first one is devoted to give frame bounds estimates for the finite-dimensional case of dynamical sampling. Then, in the second subsection, we apply these results to solve the problem of dynamical sampling for shift-preserving operators.

\subsection{Finite-dimensional dynamical sampling} \label{sec-bounds}

Along this subsection we will assume that $\mathcal H$ is an $n$-dimensional complex Hilbert space and $R:\mathcal H\to \mathcal H$ is a linear transformation. Let $I=\{1,\dots,m\}$ and $K=\{0,\dots, k\}$ where $k\geq n-1$. Moreover, for $\lambda\in\mathbb C$ we will denote by $E_\lambda = \ker(R^*-\lambda \mathcal I)$ and by $P_{E_\lambda}$ the orthogonal projection of $\mathcal H$ onto $E_\lambda$. As before, $\text{spec}(R^*)$ will denote the set of eigenvalues of $R^*$.

To obtain frame bounds estimates for Theorem \ref{DS finite dimension}, we need to study each direction separately. In one direction, we have the following result.

\begin{theorem}\label{teo1}
	If $\left\{R^{j}f_{i}: i\in I, j\in K\right\}$ is a frame of $\mathcal H$ with frame bounds $A,B>0$, then for every $\lambda\in\text{spec}(R^*)$, we have that
	$\left\{P_{E_\lambda}f_{i}:i\in I\right\}$ is a frame for $E_{\lambda}$, with bounds $A/C_{\lambda}$ and $B/C_{\lambda}$, where $C_{\lambda} = \sum_{j\in K} |\lambda|^{2j}$.
\end{theorem}

\begin{proof}
	Let $\lambda\in\mathbb C$ be an eigenvalue of $R^*$ and $f\in E_\lambda$. It is suffices to observe that
	\begin{align*}
	\sum_{j\in K}\sum_{i\in I} |\langle f, R^j f_i\rangle|^2 &= \sum_{j\in K}\sum_{i\in I}  |\langle {R^*}^j f, f_i\rangle|^2\\
	&=\sum_{j\in K}\sum_{i\in I}  |\langle \lambda^j f,f_i\rangle|^2\\
	&=\sum_{j\in K} |\lambda|^{2j}\sum_{i\in I}  |\langle f, P_{E_\lambda}f_i\rangle|^2.
	\end{align*}
\end{proof}

In the other direction, more work is required. We will ask for $R$ to be a normal linear transformation in order to have a decomposition of $\mathcal H$ into an orthogonal sum of eigenspaces. The proof of the next theorem is inspired by \cite[Theorem 5.1.]{CHP}, where frame bounds estimates where provided for the infinite-dimensional dynamical sampling problem, assuming that the operator is normal and only one function is being iterated. For this, we make use of a quite technical lemma that we show below.

We will denote by $\mathcal P_k$ the Hilbert space of the polynomials with complex coefficients of degree less than or equal to $k$, provided with the inner product
\begin{equation*}\label{norm Pk}
\langle p,q \rangle = \sum_{j=0}^k c_j\overline{d_j},
\end{equation*}
where $p(z)=\sum_{j=0}^{k} c_jz^j$ and $q(z)=\sum_{j=0}^{k} d_jz^j$.
The space $\mathcal P_k^{|I|}$ is the cartesian product of $|I|$ copies of $\mathcal P_k$ endowed with the norm
$$\|(p_i)_{i\in I}\|_{\mathcal P_k^{|I|}} = \left(\sum_{i\in I} \|p_i\|_{\mathcal P_k}^2\right)^{1/2} = \left(\sum_{i\in I} \sum_{j\in K} |c_{ij}|^2\right)^{1/2}.$$

\begin{lemma}\label{bound M lemma}	
Let $\lambda_1,\dots,\lambda_r$ be all different complex numbers. We define  $T:\mathcal P_k^{|I|} \rightarrow \ell^2(I\times \{1,\dots,r\})$ as
\begin{equation}\label{T_lambda,m}
	(p_i)_{i\in I}\mapsto (p_i(\lambda_s))_{i\in I, s\in \{1,\dots,r\}},
\end{equation} 
and $M:\ell^2(I\times \{1,\dots,r\})\rightarrow \mathcal P_k^{|I|}$ as the operator given by 
\begin{equation}\label{M}
	(w_{is})_{i\in I, s\in\{1,\dots, r\}} \mapsto (p_i)_{i\in I},
\end{equation} 
where for every $i\in I$, $p_i$ denotes the Lagrange polynomial of degree $r-1$ which interpolates $(\lambda_s,w_{is})$, $s=1,\dots,r$. Then, the following statements hold:
\begin{enumerate}[\rm(i)]
	\item $T$ is surjective.
	\item $TM=\mathcal{I}.$
	\item The operator norm of $T$ and $M$ have the following bounds:
	\begin{align*}
	\|T\|&\leq \left(r\sum_{j=0}^k \beta^{2j}\right)^{1/2}\text{ and }\\
	\|M\|&\leq \left(\frac{r}{\alpha} \sum_{u=0}^{r-1}\binom{r-1}{u}^2\beta^{2u}\right)^{1/2}
	\end{align*}
	with 
	$\alpha = \min_{1\leq s\leq r} \prod_{\substack{u=1\\ u\neq s}}^{r} |\lambda_s-\lambda_u|^2
	\text{ and }\,
	\beta = \max_{1\leq s\leq r} |\lambda_s|.
	$
\end{enumerate}
\end{lemma}

\begin{proof}
	Observe that given any vector $(w_1,\dots,w_r)\in \mathbb C^r$, one can construct the Lagrange polynomial of degree $r-1$ which interpolates the points $(\lambda_1,w_1),\dots,$ $(\lambda_r,w_r)$. This proves that $T$ is surjective.
	The relation $TM=\mathcal{I}$ follows from the definition of $T$ and $M$.
	
	To estimate the norm of $T$, let $(p_i)_{i\in I}\in \mathcal P_k^{|I|}$, using that for every $s=1,\dots, r$, $|\lambda_s|\leq \beta$, we see that
	\begin{align*}
	\|T (p_i)_{i\in I}\|_{\ell^2(I\times \{1,\dots,r\})} &= \left( \sum_{i\in I}\sum_{s=1}^{r} |p_i(\lambda_s)|^2\right)^{1/2}
	=\left(\sum_{i\in I}\sum_{s=1}^{r}\left|\sum_{j=0}^{k} c_{ij}\lambda_s^j\right|^2\right)^{1/2} \\
	&\leq\left(\sum_{i\in I}\sum_{s=1}^{r}\left(\sum_{j=0}^{k} |c_{ij}||\lambda_s|^j\right)^2\right)^{1/2}\\
	&\leq\left(\sum_{i\in I}\sum_{s=1}^{r}\left(\sum_{j=0}^{k}|c_{ij}|^2\right)\left(\sum_{j=0}^{k} |\lambda_s|^{2j} \right)\right)^{1/2}\\
	&\leq \|(p_i)_{i\in I}\|_{\mathcal P_n^{|I|}} \left(r\sum_{j=0}^k \beta^{2j}\right)^{1/2}.
	\end{align*}
	
	Now, let $w=(w_{is})_{i\in I, s\in\{1,\dots, r\}}\in\ell^2(I\times \{1,\dots,r\})$.
	Recall that for every $i\in I$, given the points $(\lambda_s,w_{is})$, $s=1,\dots,r$, the Lagrange polynomial which interpolates them is given by $p_i(z) = \sum_{s=1}^{r} w_{is}B_s(z)$, 		where 
	$$B_s(z) = \prod_{\substack{u=1\\ u\neq s}}^{r} \frac{z-\lambda_u}{\lambda_s-\lambda_u},$$
	then,
	\begin{align}
		\|M(w)\|_{\mathcal{P}_k^{|I|}} &= \left\|(p_i)_{i\in I}\right\|_{\mathcal{P}_k^{|I|}}
		= \left(\sum_{i\in I} \left\|\sum_{s=1}^{r} w_{is}B_s\right\|^2_{\mathcal P_k}\right)^{1/2}\nonumber \\
		&\leq \left(\sum_{i\in I} \left(\sum_{s=1}^{r} |w_{is}|\|B_s\|_{\mathcal P_k}\right)^2\right)^{1/2} \nonumber \\
		&\leq\left(\sum_{i\in I} \left(\sum_{s=1}^{r} |w_{is}|^2\right)\left(\sum_{s=1}^{r}\|B_s\|^2_{\mathcal P_k}\right)\right)^{1/2} \nonumber \\
		&=\left(\sum_{s=1}^{r}\|B_s\|^2_{\mathcal P_k}\right)^{1/2}\|w\|_{\ell^2(I\times \{1,\dots,r\})}.\label{bound M}
	\end{align} 
	
	Furthermore, since for every $s=1,\dots, r$, $|\lambda_s|\leq \beta$, then 
	\begin{align*}
		\|B_s\|^2_{\mathcal P_k} &= \left\|\prod_{\substack{u=1\\ u\neq s}}^{r} \frac{z-\lambda_u}{\lambda_s-\lambda_u}\right\|^2_{\mathcal P_n}=\frac{1}{\prod_{u\neq s} |\lambda_s-\lambda_u|^2} \left\|\prod_{u\neq s} (z-\lambda_u)\right\|^2_{\mathcal P_n}\\
		&\leq \frac{1}{\alpha} \left(1+\left|\sum_{u\neq s} \lambda_u\right|^2 + \left|\sum_{\substack{1\leq u<v\leq r \\ u,v\neq s}} \lambda_u\lambda_v\right|^2 + \dots + \left|\prod_{u\neq s} \lambda_u \right|^2 \right)\\
		&\leq  \frac{1}{\alpha}\left(1 + (r-1)^2\beta^2 + \binom{r-1}{2}^2 \beta^{4} + \dots + \beta^{2(r-1)} \right)\\
		&= \frac{1}{\alpha}\sum_{u=0}^{r-1}\binom{r-1}{u}^2\beta^{2u}.
	\end{align*}
	
	Finally, it follows from the latter and (\ref{bound M}) that
	\begin{equation*}
		\|M\|\leq \left(  \frac{r}{\alpha}\sum_{u=0}^{r-1}\binom{r-1}{u}^2\beta^{2u}\right)^{1/2}.
	\end{equation*}

\end{proof}

\begin{remark}
	Observe that using the  equivalence between norms $\|\cdot\|_1$ and $\|\cdot\|_2$, we can obtain instead the following bound
	\begin{equation*}
		\|B_s\|^2_{\mathcal P_k} \leq \frac{1}{\alpha}\left(\sum_{u=0}^{r-1}\binom{r-1}{u}\beta^{u}\right)^2 = \frac{1}{\alpha} \left( 1+\beta\right)^{2r},
	\end{equation*}
	and therefore,
	\begin{equation*}\|M\|\leq 			\left(\frac{r}{\alpha}\right)^{1/2} \left( 1+\beta\right)^{r}.
	\end{equation*}
\end{remark}

\begin{theorem}\label{teo1converse} 
	
	Assume that $R$ is normal and let $\lambda_1,\dots,\lambda_r\in\mathbb C$ be such that $R=\sum_{s=1}^r \lambda_s P_{E_{\lambda_s}}$.
	If for every $s=1,\dots,r$, $\left\{P_{E_{\lambda_s}}f_{i}:i\in I\right\}$ is a frame for $E_{\lambda_s}$ with frame bounds $A_s,B_s>0$, then $\left\{R^{j}f_{i}: i\in I,j\in K\right\}$ is a frame for $\mathcal H$ with bounds
	$$A\left(\frac{r}{\alpha}\sum_{u=0}^{r-1}\binom{r-1}{u}^2\|R\|^{2u}\right)^{-1}\quad \text{and}\qquad  B \left(r\sum_{j=0}^k \|R\|^{2j}\right),$$
		where  
	\begin{equation}\label{alpha_lambda}
	\alpha = \min_{1\leq s\leq r} \prod_{\substack{u=1\\ u\neq s}}^{r} |\lambda_s-\lambda_u|^2>0,	
	\end{equation}
	and $A=\min_{s} A_s$, $B= \max_{s} B_s$.

\end{theorem}

Observe that the bounds estimates obtained only depend on the operator norm, the number of eigenvalues of $R$ and, in the upper bound, the number of iterations $|K|=k$.

\begin{proof}		
	Consider $U:\ell^2(I\times K)\rightarrow \mathcal H$ the synthesis operator of 
	$\left\{R^{j}f_{i}: i\in I,j\in K\right\}$ which is defined by
	$$Uc=\sum_{i\in I}\sum_{j\in K}c_{ij}R^jf_i$$
	for every $c\in\ell^2(I\times K)$.
	Observe that we have
	\begin{align}
	Uc=\sum_{i\in I}\sum_{j\in K}c_{ij}R^jf_i &= \sum_{i\in I}\sum_{j\in K}c_{ij} \sum_{s=1}^r \lambda_s^j P_{E_{\lambda_s}}f_i \nonumber \\ &=\sum_{i\in I}\sum_{s=1}^r\left(\sum_{j\in K} c_{ij}\lambda_s^j \right)P_{E_{\lambda_s}}f_i.\label{Uc}
	\end{align} 
	
	The idea of this proof will be to estimate the frame bounds through the synthesis and the analysis operator. In order to do that, we will decompose $U$ into three auxiliary operators $C$, $P$ and $T$, which we define next.

We start by defining $C$. Since $\mathcal H= E_{\lambda_1}\oplus \dots \oplus E_{\lambda_r}$, where the sums are orthogonal, and each $E_{\lambda_s}$ has a frame with the same constants $A$ and $B$, then the union of these frames, i.e. 
	\begin{equation}\label{union of projections}
	\left\{P_{E_{\lambda_s}}f_i:i\in I, s=1,\dots,r\right\},
	\end{equation} is a frame for $\mathcal H$ with constants $A$ and $B$. 
	Thus, $C$ will be the synthesis operator of \eqref{union of projections}, i.e., $C:\ell^2(I\times\{1,\dots,r\})\rightarrow \mathcal H$ and 
	$$Ca = \sum_{i\in I}\sum_{s=1}^r a_{is}P_{E_{\lambda_s}}f_i$$
	for $a\in \ell^2(I\times\{1,\dots,r\})$. 
		
	Second, define $P:\ell^2(I\times K)\rightarrow \mathcal P_k^{|I|}$ as $c=(c_{ij})_{i,j}\mapsto (p_i)_{i\in I}$ where $p_i(z)=\sum_{j=0}^k c_{ij} z^j$. 
		Note that $P$ is an isometry.
	Finally, we consider $T$ defined as in \eqref{T_lambda,m}, with $\lambda_1,\dots,\lambda_r$ the eigenvalues of $R$. 
	
	Using the operators already defined, the operator $U$ have the expression
	$$Uc=\sum_{i\in I}\sum_{s=1}^r\left( T P c\right)_{is} P_{E_{\lambda_s}}f_i = C\,TPc,$$
	and therefore, if $f\in \mathcal H$, 
	\begin{equation}\label{U*}
	\|U^*f\|^2 = \|P^*T^*C^*f\|^2=\|T^*C^*f\|^2.
	\end{equation}
	
	Given that $T$ is, by Lemma \ref{bound M lemma}, a surjective operator, then $T^{*}$ is bounded from below by a positive constant. In particular, $$\inf\left\{\|T^*a\|:a\in\ell^2(I\times \{1,\dots,r\}),\,\|a\|=1\right\} = \gamma >0.$$ Moreover, it is well known that given an operator $M$ such that $TM = \mathcal I$, then $\gamma = 1/\|M\|$ (see, for instance, \cite{Boul}). 
	Let $M$ be defined as in \eqref{M} with $\lambda_1,\dots,\lambda_r$ the eigenvalues of $R$.
    Thus, by Lemma \ref{bound M lemma} we get that
	$TM = \mathcal I$ and since $\|R\| = \beta =\max_{1\leq s\leq r} |\lambda_s|$, we have 
	$$\gamma\geq \left(  \frac{r}{\alpha}\sum_{u=0}^{r-1}\binom{r-1}{u}^2\|R\|^{2u}\right)^{-1/2}.$$ Therefore, from (\ref{U*}) we deduce that 
	$$\|U^*f\|^2\geq \gamma\|C^*f\|^2\geq \gamma A \|f\|^2\geq A \left(\frac{r}{\alpha}\sum_{u=0}^{r-1}\binom{r-1}{u}^2\|R\|^{2u}\right)^{-1} \|f\|^2,$$
	which gives us a lower frame bound for $\left\{R^{j}f_{i}: i\in I,j\in K\right\}$. On the other hand, by the estimate bound of $\|T\|$ obtained in Lemma \ref{bound M lemma} it follows that 
		\begin{align*}
	\|U\|^2=\|CTP\|^2\leq B \|T\|^2\leq B \left(r\sum_{j=0}^k \|R\|^{2j}\right),
	\end{align*}
	from which we obtain an upper frame bound for $\left\{R^{j}f_{i}: i\in I,j\in K\right\}$.
\end{proof}	

\subsection{Dynamical sampling for shift-preserving operators}\label{section-DS-SP}

In this subsection we assume that $V$ is a finitely generated shift-invariant space of length $\ell$.
We show the solution of the dynamical sampling problem for a bounded, normal, shift-preserving operator $L:V\to V$ acting on a set of functions
$\mathcal F=\left\{ f_i:\,i\in I \right\}$ of $V$. That is, we give conditions on $L$ and $\mathcal F$ in order that the system
$$\left\{L^j f_i : i\in I \,; j\in K \right\}$$
is a frame generator set
of $V$, where $I=\{1,\dots,m\}$ and $K=\{0,\dots,\ell-1\}$. 

The conditions found are reminiscent of the conditions for the finite-dimensional case. These are obtained through the theory of $s$-diagonalization of shift-preserving operators, which was developed in \cite{ACCP} and reviewed in the previous section.

We prove two theorems, one for the necessary conditions and one for the sufficient conditions. In the first theorem, the assumption of normality of the shift-preserving operator is not needed.

\begin{theorem}\label{DS-SIS-theorem1}
	Let $V$ be a finitely generated shift-invariant space of length $\ell$, and let  $L:V\rightarrow V$ be a bounded, shift-preserving operator with range operator $R$. Suppose $I=\{1,\dots,m\}$ and $K=\{0,\dots,\ell-1\}$. Let $\{f_i:\,i\in I\}$ be a set of functions in $V$.
	
	If $\left\{L^{j}f_{i}:i\in I,\, j\in K\right\}$ is a frame generator set of $V$ with bounds $A,B>0$, then for every $s$-eigenvalue $\Lambda_{a}$ of $L^{*}$, if $V_a=\ker(L^*-\Lambda_a)$, we have that $\left\{P_{V_{a}} f_{i}: i\in I\right\}$ is a frame generator set of $V_{a}$ with frame bounds 
	$$A\left(\sum_{j=0}^{\ell-1} \|L\|^{2j} \right)^{-1}\quad \text{and}\qquad  B.$$
\end{theorem}

\begin{proof}
	Suppose that $\left\{L^{j}f_{i}:i\in I, j\in K\right\}$ is a frame generator set of $V$ with frame bounds $A,B>0$. Then, by Theorem \ref{riesz basis fiber}, 
	$$\left\{\mathcal{T}(L^{j}f_{i})(\omega): i\in I, j\in K\right\}$$ 
	is a frame of $J(\omega)$ with uniform bounds $A,B$ for a.e. $\omega\in [0,1)^{d}$. By \eqref{prop of R} we have that 
	$$ \left\{R(\omega)^{j}(\mathcal{T}f_{i}(\omega)): i\in I, \,j\in K\right\}$$
	is a frame of $J(\omega)$ with uniform bounds $A,B$ for a.e. $\omega\in [0,1)^{d}$. 
	
	Let $\Lambda_a$ be an $s$-eigenvalue of $L^*$. By Theorem \ref{adjoint L} and item (i) of Proposition \ref{prop-eigen}, we have that $\widehat{a}(\omega)$ is an eigenvalue of $R^{*}(\omega)$ for a.e. $\omega\in \sigma(V_{a})$. Then, Theorem \ref{teo1} implies that the set
	\begin{equation}\label{projections at w}
	\left\{P_{J_{V_{a}}(\omega)}(\mathcal{T}f_{i}(\omega)): i\in I, \,j\in K\right\}
	\end{equation}
	is a frame of $J_{V_{a}}(\omega)$ with bounds $A/C_{a}(\omega)$ and $B/C_{a}(\omega)$ for a.e. $\omega\in \sigma(V_{a})$, where $C_{a}(\omega)=\sum_{j=0}^{\ell-1} |\widehat{a}(\omega)|^{2j}$. 
	
	Note that by \eqref{norm of R and L} we have that
	\begin{align*}
	1\leq C_{a}(\omega)
	=\sum_{j=0}^{\ell-1} |\widehat{a}(\omega)|^{2j}
	\leq \sum_{j=0}^{\ell-1} \|R(\omega)\|^{2j}
	\leq  \sum_{j=0}^{\ell-1} \|L\|^{2j}.
	\end{align*}
	Then, we are able to obtain uniform bounds for \eqref{projections at w} for a.e. $\w\in\sigma(V_a)$.
	Finally, using Lemma \ref{Helson projections} and Theorem \ref{riesz basis fiber}, we get that $\left\{P_{V_{a}} f_{i}: i\in I\right\}$ is a frame generator set of $V_{a}$ with lower bound $ A\left(\sum_{j=0}^{\ell-1} \|L\|^{2j} \right)^{-1}$ and $B$ as an upper bound.
\end{proof}

Different from the finite-dimensional case of the dynamical sampling problem, for a sufficient condition we need to add an extra uniformity assumption.

\begin{definition}
	Let $V$ be a finitely generated shift-invariant space, and let  $L:V\rightarrow V$ be a bounded shift-preserving operator with range operator $R$. We say that $L$ has the {\it spectral gap property} if there exists a constant $c>0$ such that $|\lambda-\lambda'|\geq c$ for every $\lambda\neq \lambda'$ in $\text{spec}(R(\w))$ and for a.e. $\w\in \sigma(V)$.
\end{definition}

\begin{theorem}\label{DS-SIS-theorem2}
	Let $V$ be a finitely generated shift-invariant space of length $\ell$, and let  $L:V\rightarrow V$ be a bounded, normal, shift-preserving operator with range operator $R$, satisfying the spectral gap property.
	Let $I=\{1,\dots,m\}$, $K=\{0,\dots,\ell-1\}$ and let $\{f_i:\,i\in I\}$ be a set of functions in $V$. 
	
	Then,  $\left\{L^{j}f_{i}:i\in I,\,j\in K \right\}$ is a frame generator set of $V$ if and only if for every $s$-eigenvalue $\Lambda_a$ of $L^*$, if $V_a=\ker(L^*-\Lambda_a)$, we have that $\left\{P_{V_{a}} f_{i}: i\in I \right\}$ is a frame generator set of $V_{a}$ with the same frame bounds.
	
	\end{theorem}

\begin{proof}
	The necessity has been proved in Theorem \ref{DS-SIS-theorem1}. 
	
	For the converse, recall that since $L$ is a normal operator, then $L^*$ is $s$-dia\-gon\-alizable (see Proposition \ref{adjoint diag}). As discussed in Remark \ref{Teo 6.13}, there exists an $s$-diagonalization of $L^*$ $(V,L^*,a_1,\dots,a_r)$ such that $\sigma(V_{a_{s+1}})\subseteq\sigma(V_{a_{s}})$ for every $s=1,\dots,r-1$ and, given that $L$ has the spectral gap property, there exists a constant $c>0$ such that
	\begin{equation}\label{separation of a_s}\big|\hat{a}_s(\w)-\hat{a}_u(\w)\big|\geq c \end{equation}
	for a.e. $\w\in[0,1)^d$ and for every $u\neq s$, with $u,s=1,\dots,r$. 	
	
	Now, suposse that $\left\{P_{V_{a_s}} f_{i}: i\in I\right\}$ is a frame generator set of $V_{a_s}$ for every $s=1,...,r$ with bounds $A,B>0$. Then, by Theorem \ref{riesz basis fiber}, 
	$$\left\{\mathcal{T}(P_{V_{a_s}}f_{i})(\omega): i\in I \right\}$$
	is a frame of $J_{V_{a_s}}(\omega)$ with uniform bounds $A$, $B$ for a.e. $\omega\in [0,1)^{d}$ and for every $s=1,...,r$. 
	By Lemma \ref{Helson projections}, we have that $$\left\{P_{J_{V_{a_s}}(\omega)}(\mathcal{T}f_{i}(\omega)): i\in I\right\}$$ is a frame of $J_{V_{a_s}}(\omega)$ with uniform bounds $A$, $B$ for a.e. $\omega\in [0,1)^{d}$ and for every $s=1,...,r$. 
	
	Recall that the number of eigenvalues of $R(\w)$ may vary through the different values of $\w$. In order to apply Theorem \ref{teo1converse}, we need to split $[0,1)^d$ into measurable sets where the number of eigenvalues is constant. For that, as in Remark \ref{Teo 6.13}, let us define for $h=1,\dots,r-1$, the sets $B_h:=\sigma(V_{a_h})\setminus \sigma(V_{a_{h+1}})$ and $B_r:=\sigma(V_{a_r})$.  Then, we have that each $B_h$ is the set of all $\w\in \sigma(V)$ for which $R(\w)$ has exactly $h$ different eigenvalues. Furthermore, 
	 $$[0,1)^d = \bigcup_{h=1}^r B_h \cup \left( [0,1)^d\setminus \sigma(V)\right).$$
	
	Fix $h\in\{1,\dots,r\}$. For a.e. $\w\in B_h$, we have that $\left\{P_{J_{V_{a_s}}(\omega)}(\mathcal{T}f_{i}(\omega)): i\in I\right\}$ is a frame of $J_{V_{a_s}}(\w)$ for $s=1,\dots,h$. Theorem \ref{teo1converse} implies that
	\begin{equation}\label{frame Rj}
	\left\{R(\omega)^{j}(\mathcal{T}f_{i}(\omega)): i\in I,\, j\in K\right\}
	\end{equation}
	is a frame of $J(\omega)$ for a.e. $\omega\in B_h$ with frame bounds $$A\left( \frac{r}{\alpha(\omega)}\sum_{u=0}^{h-1} {h-1 \choose u}^{2} \|R(\omega)\|^{2u} \right)^{-1} \qquad \text{ and } \qquad B\left(h\sum_{j=0}^{\ell-1} \|R(\w)\|^{2j}\right),$$
	$$\alpha(\omega)=\min_{1\leq s\leq h} \prod_{\substack{u=1\\ u\neq s}}^{h} |\widehat{a}_{s}(\omega) - \widehat{a}_{u}(\omega)|^{2}.$$ 
	
	Observe that since \eqref{separation of a_s} holds, we might as well asume that $c<1$, and so $\alpha(\w) \geq c^{2h}\geq c^{2r}$ for a.e. $\w\in B_h$. On the other hand, by \eqref{norm of R and L} we know that $\|R(\w)\|\leq \|L\|$ almost everywhere. In this way, we are able to obtain the following uniform bounds for a.e. $\w\in B_h$,
	\begin{equation*}A\left( \frac{h}{\alpha(\omega)}\sum_{u=0}^{h-1} {h-1 \choose u}^{2} \|R(\omega)\|^{2u} \right)^{-1} \geq A\left(  \frac{r}{c^{2r}} \sum_{u=0}^{r-1} {r-1 \choose u}^2 \|L\|^{2u}\right)^{-1}\end{equation*}
	and
	\begin{equation*}
	B\left(h\sum_{j=0}^{\ell-1} \|R(\w)\|^{2j}\right) \leq B\left(r\sum_{j=0}^{\ell-1} \|L\|^{2j}\right).	
		\end{equation*}\\
	Since these bounds are the same on every $B_h$ for $h=1,\dots, r$, we have that (\ref{frame Rj}) is a frame of $J(\omega)$ for a.e. $\w\in[0,1)^d$ with uniform bounds and therefore $$\left\{L^{j}f_i: i\in I, \,j\in K\right\}$$ is a frame generator set of $V$ with bounds $A\left(  \frac{r}{c^{2r}} \sum_{u=0}^{r-1} {r-1 \choose u}^2 \|L\|^{2u}\right)^{-1}$ and \break$B\left(r\sum_{j=0}^{\ell-1} \|L\|^{2j}\right)$.
\end{proof}

\bibliographystyle{amsalpha}

\end{document}